\documentclass[letterpaper, 10 pt, conference]{ieeeconf}  

\IEEEoverridecommandlockouts                              
\usepackage{graphics} 
\usepackage{epsfig} 
\usepackage{amsmath} 
\usepackage{amssymb}  
\usepackage[dvipsnames]{xcolor}

\def\balpha{\boldsymbol{\alpha}}

\newcommand\bX{\boldsymbol{X}}
\newcommand\bY{\boldsymbol{Y}}

\newcommand\bZ{\boldsymbol{Z}}

\newcommand\bW{\boldsymbol{W}}

\newcommand\EE{\mathbb E}

\newcommand\RR{\mathbb R}

\newcommand\bbA{\mathbb A}

\DeclareMathOperator*{\argmin}{arg\,min}

\newtheorem{theorem}{Theorem}[section]

\newtheorem{definition}[theorem]{Definition}

\newtheorem{proposition}[theorem]{Proposition}

\newtheorem{assumption}[theorem]{Assumption}

\title{\LARGE \bf
 How can the tragedy of the commons be prevented?:\\ Introducing Linear Quadratic Mixed Mean Field Games
}

\author{G\"ok{\c c}e Dayan{\i}kl{\i} and Mathieu Lauri{\`e}re 
\thanks{Department of Statistics,
  University of Illinois at Urbana-Champaign, 
  Champaign, IL 61820, USA 
        {\tt\small gokced@illinos.edu}}%
\thanks{NYU-ECNU Institute of Mathematical Sciences at NYU Shanghai; Shanghai Frontiers Science Center of Artificial Intelligence and Deep Learning; NYU Shanghai, 567 West Yangsi Road, Shanghai, 200126, People’s Republic of China
        {\tt\small mathieu.lauriere@nyu.edu}}%
}

\begin{document}

\maketitle
\thispagestyle{empty}
\pagestyle{empty}

\begin{abstract}

In a regular mean field game (MFG), the agents are assumed to be insignificant, they do not realize their effect on the population level and this may result in a phenomenon coined as the Tragedy of the Commons by the economists. However, in real life this phenomenon is often avoided thanks to the underlying altruistic behavior of (all or some of the) agents. Motivated by this observation, we introduce and analyze two different mean field models to include altruism in the decision making of agents. In the first model, mixed individual MFGs, there are infinitely many agents who are partially altruistic (i.e., they behave partially cooperatively) and partially non-cooperative. In the second model, mixed population MFGs, one part of the population behaves cooperatively and the remaining agents behave non-cooperatively. Both models are introduced in a general linear quadratic framework for which we characterize the equilibrium via forward backward stochastic differential equations. Furthermore, we give explicit solutions in terms of ordinary differential equations, and prove the existence and uniqueness results.

\end{abstract}

\section{Introduction}

Mean field games (MFG) and mean field control (MFC) offer frameworks to analyze decision-making in large populations by using a mean-field approximation. In these problems, instead of considering interactions between a finite number individuals (which increases exponentially with the number of individuals), we consider the interactions between one representative agent and a mean-field, which is the population's distribution\footnote{This can be state, control or joint state-control distribution of the population.}. On the one hand, MFGs delve into non-cooperative settings, where individuals make choices based on their own benefit and the average behavior of everyone else. This corresponds to a Nash equilibrium. On the other hand, MFC considers a cooperative situation, where agents jointly optimize an objective function which represents an average over the whole population of individual's costs. Alternatively, this situation can be viewed as an optimization problem for a central (i.e., social) planner. To be more specific, the notions of solutions studied respectively in MFGs and MFC are Nash equilibrium and social optimum.

So far the two settings have been extensively studied, both theoretically and numerically. We refer to~\cite{Bensoussan_Book,CarmonaDelarue_book_I,CarmonaDelarue_book_II} for more information. For these two settings many applications have been proposed, such as finance~\cite{carmona2015mean,cardaliaguet2018mean,carmona2021applications,carmona2023deep}, economics~\cite{achdou2014partial,achdou2022income,advertisement}, epidemic management~\cite{elie2020contact,aurell2022optimal,olmez2022modeling}, cybersecurity~\cite{kolokoltsov2016mean}, or energy production and climate change~\cite{gueant2011mean,chan2017fracking,alasseur2020extended,carbon_StackelbergMFG,dayanikli2024multipopulation}, to cite just a few. 

These two settings correspond to extreme cases: the agents are either fully non-cooperative (selfish) or fully cooperative (altruistic). However, many real life applications do not fall in one of these situations because there is a mix of cooperation and competition. The main focus of this paper is to study models which combine cooperative and non-cooperative behaviors in the context of mean field populations of agents.

One of our motivations for combining the cooperative and non-cooperative behaviors is the \textit{tragedy of the commons}, which describes a situation where individuals, acting in their own self-interest, overuse a shared resource, ultimately exploiting it to create a long term negative outcome for everyone. For instance shepherds might add more sheep to a pasture for personal gain, but collectively they risk depleting the grass, harming the land, and hurting everyone's livelihood in the long run. However, in many cases, the agents anticipate this catastrophic outcome and behave in a sufficiently altruistic (i.e. cooperative) way to avoid exploiting common resources~\cite{ostrom1990governing,ostrom1999coping}. 

The literature on the models that includes both cooperative and non-cooperative behavior in the mean field models is not yet well developed. The closest analysis to our equilibrium notions can be found in~\cite{carmona2023nash}; however, the models differ since in~\cite{carmona2023nash} the cooperative agents do not take it into account the behavior of the non-cooperative agents in their optimization. A related setup is given in~\cite{angiuli2023reinforcementmixedmfcg} where the authors study an extension of MFGs where each
agent solves an MFC which can be seen as the limiting scenario for a competition between a large number of large coalitions. In~\cite{Barreiro_Gomez_Duncan_Tembine_2020}, authors study co-opetitive linear quadratic mean field games in which agents take into account the other agent’s
costs positively or negatively while making decisions. Recently,~\cite{guo2023mesob} studied a bi-level optimization problem to balance equilibrium and social optimum. In this paper, we focus on two types of mean field models to explore new equilibrium notions in the cases where both cooperative and non-cooperative behaviors can be prevalent. In the first type, there is a single group of agents, in which every agent's objective function contains terms which model individuals' own cooperative and non-cooperative behaviors. In the second type, there are two sub-groups in the population, one in which all the agents are cooperative and the other one in which all the agents are non-cooperative. Finally, there is also literature on mean field models that models multi-population settings in which the agents of each population are either non-cooperative (commonly referred as multi-population MFG)~\cite{achdou2017mean,cirant2015multi,dayanikli2024multipopulation} or cooperative (commonly referred as mean field type games)~\cite{djehiche2017mean,barreiro2021mean}. In this paper, we consider linear-quadratic models. Such models are popular for their tractability, including in the MFG and MFC setting~\cite{huang2007large,bardi2012explicit,bensoussan2016linear,delarue2020selection,wang2020mean}. The rest of the paper is organized as follows. In Section~\ref{sec:finite-pop}, we introduce the finite population models. In Section~\ref{sec:mean-field-models}, we present the corresponding mean field models by taking the number of agents to infinity. The main results are in Section~\ref{sec:main-results}: we first show that, in each case, the solution can be characterized by a system of forward-backward stochastic differential equations (FBSDEs) of McKean-Vlasov type by following Pontryagin stochastic maximum principle. We then show that these FBSDE systems can be reduced to systems of forward-backward ordinary differential equations (FBODEs). Finally we give the existence and uniqueness results for the first model and the existence results for the second model.

\section{Finite population version}
\label{sec:finite-pop}

We consider a finite time horizon $T>0$.  We will use bold letters to denote functions of time. 
To alleviate the notations, we will restrict the presentation to one-dimensional states and actions but the ideas and the theoretical analysis can be generalized to the multi-dimensional case in a straightforward way. We will consider $\RR$-valued open-loop controls that are progressively measurable processes, adapted to all the available information, and square integrable. We denote by $\mathbb{A}$ the set of such controls. We will denote by $\EE$ the expectation of a random variable.

\subsection{Mixed individual}

We consider a population of $N$ non-cooperative agents that are indistinguishable. 
We will denote by $\balpha^{i} = (\alpha^{i}_t)_{t \in [0,T]}$ the control used by agent $i \in [N]$ where $[N] :=\{1,2, \dots, N\}$.

Assume the controls used by the agents other than agent $i$ are given and denoted by $\underline\balpha^{-i} = (\balpha^{1}, \dots, \balpha^{i-1}, \balpha^{i+1},\dots,\balpha^{N})$. The state of agent $i \in [N]$ at time $t$ is denoted by $X^{i, \balpha^{i }}_t \in \RR$ when using control $\balpha^{i }$ and its dynamics are: 
\small
\begin{equation*}
\begin{aligned}
    dX^{i, \balpha^{ i}}_t =&\Big(b_\alpha \alpha^{i}_t + b_X X^{i, \balpha^{ i}}_t\\
    &+ b_\mu  \big(\lambda \bar X_t + (1-\lambda) \EE[X^{i, \balpha^{i}}_t]\big) \Big) dt + \sigma dW^{ i}_t, 
\end{aligned}
\end{equation*}
\normalsize
where $b_\alpha, b_X, b_\mu, \sigma \in \RR$ are constant coefficients,  $\bW^{i}$ is a Brownian motion representing idiosyncratic noise, independent of all the other sources of randomness, and $X_0^i\sim\mu_0$ where $\mu_0$ is the initial distribution. Here and thereafter, the expectation of agent $i$'s state is denoted by $\EE[X^{i, \balpha^{i}}_t]$, while the mean of the states is denoted by:
    $\bar{X}_t
    = \frac{1}{N} \sum_{j=1}^{N} X^{j,\balpha^{j}}_t. $
    We stress that the evolution of $X^{i,\balpha^{i}}_t$ depends on the controls used by all the agents through the $\bar{X}_t$. 

The (non-cooperative) agent $i \in [N]$ aims to minimize the following cost over $\balpha^{i}$, while $\underline\balpha^{-i}$ is given:
\small
\begin{align*}
        &J(\balpha^{i}; \underline\balpha^{-i}) =\EE\Bigg[\int_0^T \Big(\frac{c_\alpha}{2} (\alpha^{i}_t)^2 + \frac{c_X}{2}(X^{i, \balpha^{i}}_t)^2 
        \\
        &\qquad + \frac{c_\mu}{2} \left( \lambda \bar X_t^2 + (1-\lambda) \EE[X^{i, \balpha^{i}}_t]^2 \right)  
        \Big) dt  + \frac{c_T}{2} (X^{i, \balpha^{i}}_T)^2\Bigg].
\end{align*}
\normalsize
Here, $c_\alpha, c_X, c_\mu, c_T>0$ are constant coefficients. The first and second term respectively penalize large (in absolute value) individual actions and states. The third term penalizes large means. Here, and $\lambda \in[0,1]$ gives the level of altruism of the representative agent. Note that when $N$ is large, agent $i$ has only a negligible influence on the empirical average $\bar X_t$, so removing it from the cost function would not modify her optimal control. However, agent $i$ has an influence on her own state's mean $\EE[X^{i, \balpha^{i}}_t]$. The last term is a terminal cost which penalizes large terminal state.

\begin{definition}[Mixed individual equilibrium]
    Let $\epsilon>0$. An $\epsilon$-Nash equilibrium for the mixed individual equilibrium is a control profile $\underline{\hat\balpha} = (\hat\balpha^{j})_{j=1,\dots,N}$ such that: 
    for every $i \in [N]$, $\hat\balpha^{i}$ is an $\epsilon$-minimizer for $J(\cdot; \underline{\hat\balpha}^{-i})$. 
    A Nash equilibrium is an $\epsilon$-equilibrium with $\epsilon=0$. 
\end{definition}

\subsection{Mixed population}

We consider two groups, with respectively $N^{{\rm NC}}$ and $N^{{\rm C}}$ agents. The agents of the first group are non-cooperative, while the agents in the second group are cooperative with agents of the same sub-group. Let $N = N^{{\rm NC}} + N^{{\rm C}}$ be the total number of agents and let $p = \frac{N^{{\rm NC}}}{N}$ be the proportion of non-cooperative agents.

We will denote by $\balpha^{{\rm NC},i} = (\alpha^{{\rm NC},i}_t)_{t \in [0,T]}$ the control used by non-cooperative agent $i \in [N^{\rm NC}]$ and by $\balpha^{{\rm C},i} = (\alpha^{{\rm C},i}_t)_{t \in [0,T]}$ the control used by cooperative agent $i \in [N^{\rm C}]$.

\textbf{Non-cooperative agents. } Assume the controls $\underline\balpha^{{\rm NC},-i} = (\balpha^{{\rm NC},1}, \dots, \balpha^{{\rm NC},i-1}, \balpha^{{\rm NC},i+1},\dots,\balpha^{{\rm NC},N^{\rm NC}})$ used by the other non-cooperative agents and the controls $\underline\balpha^{{\rm C}} = (\balpha^{{\rm C},1}, \dots, \balpha^{{\rm C},N^{\rm C}})$ used by the cooperative agents are given. The state of non-cooperative agent $i \in [N^{\rm NC}]$ at time $t$ is denoted by $X^{i, \balpha^{{\rm NC},i }}_t \in \RR$ when using control $\balpha^{{\rm NC},i }$ and its dynamics are: 
\small
\begin{equation*}
\begin{aligned}
    dX^{\balpha^{{\rm NC}, i}}_t =&\Big(b^{{\rm NC}}_\alpha \alpha^{{\rm NC}, i}_t + b^{{\rm NC}}_X X^{i, \balpha^{{\rm NC}, i}}_t \\
    &+ b^{{\rm NC}}_\mu  \big(p\bar{X}^{{\rm NC}}_t + (1-p) \bar{X}_t^{{\rm C}}\big)\Big) dt + \sigma^{\rm NC} dW^{{\rm NC}, i}_t,
\end{aligned}
\end{equation*}
\normalsize
where $b^{{\rm NC}}_\alpha, b^{{\rm NC}}_X, b^{{\rm NC}}_\mu, \sigma^{\rm NC} \in \RR$ are constant coefficients,  $\bW^{{\rm NC}, i}$ is a Brownian motion representing idiosyncratic noise, independent of all the other sources of randomness {and $X_0^{{\rm NC},i}\sim\mu^{\rm NC}_0$ are i.i.d., where $\mu^{\rm NC}_0$ is the initial distribution. Here and thereafter, the means of the states of each subgroups are denoted by:
\small
\begin{equation*}
    \bar{X}^{{\rm NC}}_t
    = \frac{1}{N^{\rm NC}} \sum_{j=1}^{N^{\rm NC}} X^{j,\balpha^{{\rm NC}, j}}_t ,
    \quad 
    \bar{X}^{{\rm C}}_t
    = \frac{1}{N^{\rm C}} \sum_{j=1}^{N^{\rm C}} X^{j,\balpha^{{\rm C}, j}}_t. 
\end{equation*}
\normalsize
We stress that the evolution of $X^{i,\balpha^{{\rm NC}, i}}_t$ depends on the controls used by all the agents through the above averages. 

This non-cooperative agent $i \in [N^{{\rm NC}}]$ aims to minimize the following cost over $\balpha^{{\rm NC},i}$, while $\underline\balpha^{{\rm NC},-i}$ and $\underline\balpha^{{\rm C}} = (\balpha^{{\rm C},j})_{j=1,\dots,N^{\rm C}}$ are given:
\small
\begin{align*}
        &J^{\rm NC}(\balpha^{{\rm NC},i}; \underline\balpha^{{\rm NC},-i}, \underline\balpha^{{\rm C}}) 
        \\
        &=\EE\Bigg[\int_0^T \Big(\frac{c^{{\rm NC}}_\alpha}{2} (\alpha^{{\rm NC},i}_t)^2 + \frac{c^{{\rm NC}}_X}{2}(X^{i, \balpha^{{\rm NC},i}}_t)^2 
        \\
        &\qquad + \frac{c^{{\rm NC}}_\mu}{2}\left(p \bar X_t^{{\rm NC}} +(1-p)\bar{X}^{{\rm C}}_t\right)^2  
        \Big) dt  + \frac{c^{{\rm NC}}_T}{2} (X^{i, \balpha^{{\rm NC},i}}_T)^2\Bigg].
\end{align*}
\normalsize
Here, $c^{{\rm NC}}_\alpha, c^{{\rm NC}}_X, c^{{\rm NC}}_\mu, c^{{\rm NC}}_T>0$ are constant coefficients. The first and second term respectively penalize large (in absolute value) individual actions and states. The third term penalizes large means. Note that, when $N^{\rm NC}$ is large, non-cooperative agent $i$ has only a negligible influence on this term, so removing it from the cost function would not modify her optimal control. However, this term will have an influence for the cooperative agents so we include it here too. The last term is a terminal cost which penalizes large terminal state.

{\bf Cooperative agents. } The state of cooperative agent $i \in [N^{\rm C}]$ at time $t$ is denoted by $X^{i, \balpha^{{\rm C},i }}_t \in \RR$ and its dynamics are:
\small
\begin{equation*}
\begin{aligned}
    dX^{\balpha^{{\rm C}, i}}_t =& \Big(b^{{\rm C}}_\alpha \alpha^{{\rm C}, i}_t + b^{{\rm C}}_X X^{i, \balpha^{{\rm C}, i}}_t\\
    &+b^{{\rm C}}_\mu  \big(p\bar{X}^{{\rm NC}}_t + (1-p) \bar{X}_t^{{\rm C}}\big)\Big) dt + \sigma^{\rm C} dW^{{\rm C}, i}_t,
\end{aligned}
\end{equation*}
\normalsize
where $b^{{\rm C}}_\alpha, b^{{\rm C}}_X, b^{{\rm C}}_\mu, \sigma^{\rm C} \in \RR$ are constant coefficients,  $\bW^{{\rm C}, i}$ is a Brownian motion representing idiosyncratic noise, independent of all the other sources of randomness and $X_0^{{\rm C},i}\sim\mu^{\rm C}_0$ i.i.d., where $\mu^{\rm C}_0$ is the initial distribution. 

Given the non-cooperative agents' controls $\underline\balpha^{{\rm NC}} = (\balpha^{{\rm NC},j})_{j=1,\dots,N^{\rm NC}}$, the cooperative agents try to jointly minimize over $\underline\balpha^{{\rm C}} = (\balpha^{{\rm C},j})_{j=1,\dots,N^{\rm C}}$ the social cost:
\small
\begin{align*}
        &J^{\rm C}(\underline\balpha^{{\rm C}}; \underline\balpha^{{\rm NC}}) = \frac{1}{N^{\rm C}} \sum_{i=1}^{N^{\rm C}}\EE\Bigg[\int_0^T \Big(\frac{c^{{\rm C}}_\alpha}{2} (\alpha^{{\rm C},i}_t)^2 + \frac{c^{{\rm C}}_X}{2}(X^{i, \balpha^{{\rm C},i}}_t)^2 
        \\
        &\qquad + \frac{c^{{\rm C}}_\mu}{2}\left(p \bar X_t^{{\rm NC}} +(1-p)\bar{X}^{{\rm C}}_t\right)^2  
        \Big) dt  + \frac{c^{{\rm C}}_T}{2} (X^{i, \balpha^{{\rm C},i}}_T)^2\Bigg].
\end{align*}
\normalsize
Here, $c^{{\rm C}}_\alpha, c^{{\rm C}}_X, c^{{\rm C}}_\mu, c^{{\rm C}}_T>0$ are constant coefficients. The terms composing the cost have the same interpretation as above. Note that the cooperative agents have a collective impact on the average $\bar{X}^{{\rm C}}_t$.

\begin{definition}[Mixed population equilibrium]
    Let $\epsilon>0$. An $\epsilon$-Nash equilibrium for the mixed population equilibrium is a pair of a control profile $\underline{\hat\balpha}^{{\rm NC}} = (\hat\balpha^{{\rm NC},j})_{j=1,\dots,N^{\rm NC}}$ for the non-cooperative agents and a control profile $\underline{\hat\balpha}^{{\rm C}} = (\hat\balpha^{{\rm C},j})_{j=1,\dots,N^{\rm C}}$ for the cooperative agents such that:
    \begin{itemize}
        \item for every $i \in [N^{\rm NC}]$, $\hat\balpha^{{\rm NC},i}$ is an $\epsilon$-minimizer for $J^{\rm NC}(\cdot; \underline{\hat\balpha}^{{\rm NC},-i}, \underline{\hat\balpha}^{{\rm C}})$
        \item $\underline{\hat\balpha}^{{\rm C}}$ is an $\epsilon$-minimizer $J^{\rm C}(\cdot; \underline{\hat\balpha}^{{\rm NC}})$.
    \end{itemize}
    A Nash equilibrium is an $\epsilon$-equilibrium with $\epsilon=0$. 
\end{definition}

\section{Mean field version}
\label{sec:mean-field-models}

In this section, we introduce the corresponding mean field models for the mixed individual and mixed population setups.

\subsection{Mixed individual}
We assume that $N\rightarrow\infty$. Since every agent is identical, we can focus on a representative agent. The representative agent aims to minimize the following cost by choosing a square-integrable, progressively measurable control process $\balpha =(\alpha_t)_{t\in[0,T]}$ where $\alpha_t \in \RR$:
\small
\begin{equation}
\begin{aligned}
    &J(\balpha; \bar \bX):=\EE\Bigg[\int_0^T \Big(\frac{c_\alpha}{2} \alpha_t^2 + \frac{c_X}{2}(X^{\balpha}_t)^2 + \\
    &\frac{c_\mu}{2}\left(\lambda (\bar X_t)^2 +(1-\lambda)(\bar{X}^{\balpha}_t)^2\right)\Big) dt  + \frac{c_T}{2} (X_T^{\balpha})^2\Bigg],
\end{aligned}
\end{equation}
\normalsize
where $\bar{X}^{\balpha}_t = \int_{\RR} x d\mu_t^{\balpha}(x)$. Here, $c_\alpha, c_X, c_\mu, c_T>0$ are constant coefficients and we recall that $\lambda \in[0,1]$ gives the level of altruism of the representative agent. When $\lambda = 0$, the agents in a pure game theoretical (i.e., non-cooperative) setup; whereas, when $\lambda=1$, the agents are in a control (i.e., cooperative) setup. Therefore, this setting can be understood as an interpolation between MFG and MFC problems.
The dynamics of the representative agent's state process $\bX^{\balpha} = (X^{\balpha}_t)_{t\in[0,T]}$ where $X_t\in \RR$ is given as:
\small
\begin{equation*}
\begin{aligned}
    dX^{\balpha}_t = \big(b_\alpha \alpha_t +b_X X_t^{\balpha}+ b_\mu  \big(\lambda \bar{X}_t + (1-\lambda) \bar{X}_t^{\balpha}\big)\big) dt + \sigma dW_t,
\end{aligned}
\end{equation*}
\normalsize
where $b_\alpha, b_X, b_\mu \in \RR$ are non-zero constant coefficients, $\bW$ is the Brownian motion representing the idiosyncratic noise, $X_0 \sim \mu_0$, and $\lambda$ is interpreted as introduced above.

\begin{definition}[Mixed Individual MFNE]\label{def:mixed_individual_mfg_nash} We will call a control and mean field tuple $(\hat{\balpha}, \hat{\bar \bX})$ a mixed individual mean field Nash equilibrium (MI-MFNE) if:
\begin{itemize}
    \item[i.] $\hat \balpha$ is the best response of the representative agent given the mean field $\hat{\bar\bX}$. In other words, $\hat{\balpha} \in \argmin_{\balpha\in \bbA} J(\balpha;\hat{\bar{\bX}})$,
    \item[ii.] For all $t\in[0,T]$, we have $\hat{\bar X}_t = \bar{X}^{\hat{\balpha}}_t$.
\end{itemize}
\end{definition}

\subsection{Mixed population}

We assume $N\rightarrow\infty$ and the population is mixed, i.e., a proportion $p$ of the agents are non-cooperative and the remaining $(1-p)$ proportion of the agents are cooperative. For the mixed population mean field model, we need to introduce both the model of the representative non-cooperative agent and of the representative cooperative agent. 

The representative \textbf{non-cooperative} agent aims to minimize the following cost over the square integrable and progressively measurable control process $\balpha^{\rm NC} =(\alpha^{\rm NC}_t)_{t\in[0,T]}$ where $\alpha^{\rm NC}_t \in \RR$:
\small
\begin{equation}
\begin{aligned}
    &J^{\rm NC}(\balpha^{\rm NC}; \bar{X}^{\rm NC}, \bar{X}^{\rm C}):=
    \\
    &\EE\Bigg[\int_0^T \Big(\frac{c^{\rm NC}_\alpha}{2} (\alpha^{\rm NC}_t)^2 + \frac{c^{\rm NC}_X}{2}(X^{\balpha^{\rm NC}}_t)^2 \\
    &+\frac{c^{\rm NC}_\mu}{2}\left(p \bar X_t^{\rm NC} +(1-p)\bar{X}^{\rm C}_t\right)^2  
        \Big) dt  + \frac{c^{\rm NC}_T}{2} (X^{\balpha^{\rm NC}}_T)^2\Bigg].
    \end{aligned}
\end{equation}
\normalsize
Here, $c^{\rm NC}_\alpha, c^{\rm NC}_X, c^{\rm NC}_\mu, c^{\rm NC}_T>0$ are constant coefficients. The dynamics of the representative non-cooperative agent's state process $\bX^{\balpha^{\rm NC}} = (X^{\balpha^{\rm NC}}_t)_{t\in[0,T]}$ where $X^{\balpha^{\rm NC}}_t\in \RR$ is:
\small
\begin{equation*}
\begin{aligned}
    dX^{\balpha^{\rm NC}}_t =& \Big(b^{\rm NC}_\alpha \alpha^{\rm NC}_t + b^{\rm NC}_X X_t^{\rm NC} \\
    &+ b^{\rm NC}_\mu  \big(p\bar{X}^{\rm NC}_t + (1-p) \bar{X}_t^{\rm C}\big)\Big) dt + \sigma dW^{\rm NC}_t,
\end{aligned}
\end{equation*}
\normalsize
where $b^{\rm NC}_\alpha, b^{\rm NC}_X, b^{\rm NC}_\mu \in \RR$ are non-zero constant coefficients, $\bW^{\rm NC}$ is the Brownian motion representing the idiosyncratic noise for non-cooperative agents, $X^{\rm NC}_0 \sim \mu^{\rm NC}_0$, and $p$ is interpreted as introduced above.

Similarly, the representative \textbf{cooperative} agent aims to minimize the following cost over the square integrable and progressively measurable control process $\balpha^{\rm C} =(\alpha^{\rm C}_t)_{t\in[0,T]}$ where $\alpha^{\rm C}_t \in \RR$:
\small
\begin{equation}
\begin{aligned}
    &J^{\rm C}(\balpha^{\rm C}; \bar{X}^{\rm NC}):= \EE\Bigg[\int_0^T \Big(\frac{c^{\rm C}_\alpha}{2} (\alpha^{\rm C}_t)^2 +\frac{c^{\rm C}_X}{2}(X^{\balpha^{\rm C}}_t)^2  
    \\
    &+\frac{c^{\rm C}_\mu}{2}\left(p \bar X_t^{\rm NC} +(1-p)\bar{X}^{\balpha^{\rm C}}_t\right)^2  
        \Big) dt  + \frac{c^{\rm C}_T}{2} (X^{\balpha^{\rm C}}_T)^2\Bigg].
    \end{aligned}
\end{equation}
\normalsize
Here, $c^{\rm C}_\alpha, c^{\rm C}_X, c^{\rm C}_\mu, c^{\rm C}_T>0$ are constant coefficients. The dynamics of the representative cooperative agent's state process $\bX^{\balpha^{\rm C}} = (X^{\balpha^{\rm C}}_t)_{t\in[0,T]}$ where $X^{\balpha^{\rm C}}_t\in \RR$ is given as:
\small
\begin{equation*}
\begin{aligned}
    dX^{\balpha^{\rm C}}_t =& \Big(b^{\rm C}_\alpha \alpha^{\rm C}_t + b^{\rm C}_X X_t^{\rm C}\\
    &+ b^{\rm C}_\mu  \big(p\bar{X}^{\rm NC}_t + (1-p) \bar{X}_t^{\balpha^{\rm C}}\big)\Big) dt + \sigma dW^{\rm C}_t,
\end{aligned}
\end{equation*}
\normalsize
where $b^{\rm C}_\alpha, b^{\rm C}_X, b^{\rm C}_\mu \in \RR$ are non-zero constant coefficients, $\bW^{\rm C}$ is the Brownian motion representing the idiosyncratic noise for non-cooperative agents, $X^{\rm C}_0 \sim \mu^{\rm C}_0$, and $p$ is interpreted as introduced above.

\begin{definition}[Mixed Population MFE]\label{def:mixedpop_mfg_nash}
We will call controls and mean field tuple $(\hat{\balpha}^{\rm NC}, \hat{\balpha}^{\rm C}, \hat{\bar{\bX}}^{\rm NC})$ a mixed population mean field equilibrium (MP-MFE) if:
    \begin{itemize}
        \item[i.] $\hat{\balpha}^{\rm NC}$ is the best response of the representative non-cooperative agent given the mean fields of the non-cooperative and cooperative agents, $\hat{\bar{\bX}}^{\rm NC}$ and ${\bar{X}}^{\hat{\balpha}^{\rm C}}$, respectively. In other words, $\hat{\balpha}^{\rm NC} \in \argmin_{\balpha^{\rm NC}\in \mathbb{A}} J^{\rm NC}(\balpha^{\rm NC};\hat{\bar{\bX}}^{\rm NC},{\bar{\bX}}^{\hat{\balpha}^{\rm C}})$,
        \item[ii.] For all $t\in[0,T]$, we have $\hat{\bar{X}}_t^{\rm NC} = \bar{X}^{\hat{\balpha}^{\rm NC}}_t$,
        \item[iii.] $\hat{\balpha}^{\rm C}$ is the optimal control of the representative cooperative agent given the mean field of the non-cooperative agents, $\hat{\bar{\bX}}^{\rm NC}$. In other words, $\hat{\balpha}^{\rm C} \in \argmin_{\balpha^{\rm C}\in \mathbb{A}} J^{\rm C}(\balpha^{\rm C};\hat{\bar{\bX}}^{\rm NC})$.
    \end{itemize}
\end{definition}

We would like to stress that a mixed population mean field problem is different than multi-population MFG, in which all the agents are non-cooperative, and than mean-field type games, in which all the agent of each sub-population are cooperative with each other. It is also different than a major-minor MFG problem, see e.g.~\cite{huang2010large} and~\cite[Section 7.1]{CarmonaDelarue_book_II}. In major-minor MFG, there is only one major player (instead of a population of cooperative agents). This major agent's state directly affect the cost function and/or dynamics of the minor agents. If the state dynamics of the major agent has a noise, then this noise behaves as a common noise for the minor agents. However, in the current setup the distribution of the cooperative agents affect the competitive agents.}

\section{Main results}
\label{sec:main-results}

\subsection{Mixed individual}

\begin{theorem}[FBSDE characterization of equilibria]
    A control $\hat{\balpha}$ is an MI-MFNE control profile (see Definition~\ref{def:mixed_individual_mfg_nash}) if and only if:
    \small
    \begin{equation}
    \label{eq:mi-BR-formula}
        \hat \alpha_t
        = -\frac{b_\alpha Y_t}{c_\alpha}   
    \end{equation}
    \normalsize
    where $(\bX, \bY, \bZ)=(X_t, Y_t, Z_t)_{t\in [0,T]}$ solve the following forward-backward stochastic differential equation (FBSDE) system:
    \small
    \begin{equation}
    \label{eq:MI_FBSDE}
    \begin{aligned}
        &dX_t = \Big(-\frac{(b_\alpha)^2}{c_\alpha} Y_t +b_X X_t+ b_\mu  \bar{X}_t \Big) dt + \sigma dW_t,\ X_0 \sim \mu_0, 
        \\[1mm]
        &dY_t = -\big(b_X Y_t + c_X X_t + b_\mu(1-\lambda)\bar{Y}_t + c_\mu (1-\lambda)\bar{X}_t\big)dt\\
        &\qquad\qquad+ Z_t dW_t, \quad Y_T =c_T X_T.  
    \end{aligned}
    \end{equation}
    \normalsize
\end{theorem}

\begin{proof}
The proof is based on i) characterizing the solution of a mean field control problem with a given mean field $\bar \bX$ and ii) realizing that at the equilibrium we should have $\bar{X}_t$ is equal to the mean field introduced by mean field control, $\bar{X}_t^{\hat{\balpha}}$ for all $t\in[0,T]$.
We have the following Hamiltonian:
\small
\begin{equation*}
\begin{aligned}
    &H(t, x, \alpha, \bar x, \bar x^{\alpha}, y)= \big(b_\alpha \alpha +b_X x^{\balpha}+ b_\mu  \big(\lambda \bar{x} + (1-\lambda) \bar{x}^{\balpha}\big)\big)y \\
    &+ \frac{c_\alpha}{2} \alpha^2 + \frac{c_X}{2}(x^{\balpha})^2 + 
    \frac{c_\mu}{2}\left(\lambda (\bar x)^2 +(1-\lambda)(\bar{x}^{\balpha})^2\right)
\end{aligned}
\end{equation*}
\normalsize
and the optimal control is given as the minimizer of the Hamiltonian:
\small
\begin{equation*}
    \hat{\alpha}_t = -\frac{b_\alpha Y^{\hat\balpha}_t}{c_\alpha}
\end{equation*}
\normalsize
where $(\bX^{\hat\balpha}, \bY^{\hat\balpha}, \bZ^{\hat\balpha})$ satisfies the following FBSDE system characterizing the Mean Field Control solution given $\bar{\bX}$ (e.g.~\cite[Chapter 6.7]{CarmonaDelarue_book_I}): 
\small
\begin{equation*}
\begin{aligned}
    &dX^{\hat\balpha}_t = \Big(-\frac{(b_\alpha)^2}{c_\alpha} Y^{\hat\balpha}_t +b_X X^{\hat\balpha}_t+ \\
    &\qquad b_\mu (\lambda \bar{X}_t+(1-\lambda) \bar{X}^{\hat\balpha}_t )\Big) dt + \sigma dW_t,\ X_0 \sim \mu_0, 
    \\[1mm]
    &dY^{\hat\balpha}_t = -\big(b_X Y^{\hat\balpha}_t + c_X X^{\hat\balpha}_t + b_\mu(1-\lambda)\bar{Y}^{\hat\balpha}_t\\
    &\qquad+ c_\mu (1-\lambda)\bar{X}^{\hat\balpha}_t\big)dt
    + Z_t^{\hat\balpha} dW_t, \quad Y_T^{\hat\balpha} =c_T X_T^{\hat\balpha}. 
\end{aligned}  
\end{equation*}
\normalsize
The equation for the adjoint process is written as $dY^{\hat\balpha}_t = -\big(\partial_x H(t, X^{\hat\balpha}_t, \alpha, \bar X_t, \bar X_t^{\alpha}, Y^{\hat\balpha}_t) + \EE[\partial_{\bar x^\alpha}H(t, X^{\hat\balpha}_t, \alpha, \bar X_t, \bar X_t^{\alpha}, Y^{\hat\balpha}_t)]\big) dt+ Z_t^{\hat\balpha} dW_t$ with $Y_T^{\hat\balpha} = \partial_x \frac{c_T}{2} (X^{\hat\balpha}_T)^2$ by using the stochastic Pontryagin maximum principle. At the end, we need to impose our fixed point condition $\bar{X}_t=\bar{X}^{\hat\balpha}_t$ for all $t \in[0,T]$ to find the MI-MFNE. In order to simplify the notations, we also drop the superscript $\balpha$ and conclude our result. We would like to emphasize that different than a regular MFG, the effect of the individuals deviation on the $\bar{X}^{\balpha}$ should be taken into account, as in MFC problems. Hence, the adjoint process dynamics includes a partial derivative with respect to $\bar{x}^\alpha$.
\end{proof}

\begin{proposition}[FBODE characterization of MI-MFNE]
\label{prop:mi_FBODE}
Assume there exists an $\RR \times \RR \times \RR\times \RR$-valued function $t \mapsto (A_t, B_t, C_t, \bar{X}_t)$ solving the following forward-backward ordinary differential equation (FBODE) system:
\small
\begin{align}
        &{\dot A_t} - \frac{(b_\alpha)^2}{c_\alpha} A_t^2 + 2b_X A_t +c_X=0, \label{eq:FBODE_A} \\
        &\dot B_t  - \frac{(b_\alpha)^2}{c_\alpha} B_t^2 - 2\frac{(b_\alpha)^2}{c_\alpha} A_tB_t +(2b_X+2b_\mu -\lambda b_\mu)B_t\nonumber \\ 
        &+b_\mu(1-\lambda) A_t + c_\mu(1-\lambda)=0, \label{eq:FBODE_B}\\
       &\dot C_t  +\left(- \frac{(b_\alpha)^2}{c_\alpha} (A_t+B_t) +b_X + b_\mu(1-\lambda)\right)C_t=0, \label{eq:FBODE_C}\\
       &\dot{\bar{X}}_t = \left(- \frac{(b_\alpha)^2}{c_\alpha} (A_t+B_t) +b_X + b_\mu\right) \bar{X}_t - \frac{(b_\alpha)^2}{c_\alpha} C_t, \label{eq:FBODE_BarX}\\
       &A_T =c_T, \quad B_T = 0,\quad C_T=0, \quad \bar{X}_0 = \bar \mu_0 .\label{eq:FBODE_boundarycond}
\end{align}
\normalsize
Then 
\small
\begin{equation}
    \label{eq:equilibrium-ctrl-minors-only}
\hat \alpha_t= -\frac{b_\alpha}{c_\alpha} (A_tX_t + B_t \bar{X}_t +C_t) 
\end{equation}
\normalsize
is the MI-MFNE control. 
\end{proposition}

\begin{proof}
    Since our problem is in linear-quadratic form, we propose the following ansatz $Y_t=A_tX_t +B_t\bar{X}_t +C_t$. Then, we have:
    \small
    \begin{equation}
    \label{eq:ansatz_backward}
        \begin{aligned}
            dY_t =& \dot{A}_t X_tdt + A_t dX_t + \dot{B}_t\bar{X}_t dt + B_td\bar{X}_t+\dot{C}_tdt\\
            =& \dot{A}_t X_tdt + A_t \big(\big(-\frac{(b_\alpha)^2}{c_\alpha} Y_t +b_X X_t+ 
            b_\mu  \bar{X}_t \big) dt + \sigma dW_t \big)\\
             &
             + \dot{B}_t\bar{X}_tdt+ B_t\big(\big(-\frac{(b_\alpha)^2}{c_\alpha} \bar{Y}_t +(b_X + 
             b_\mu)  \bar{X}_t\big) dt\big) +\dot{C}_tdt\\
     =&\dot{A}_t X_tdt + A_t \big(\big(-\frac{(b_\alpha)^2}{c_\alpha} (A_tX_t +B_t\bar{X}_t +C_t) +b_X X_t\\
     &+ 
     b_\mu  \bar{X}_t \big) dt + \sigma dW_t \big)
     + \dot{B}_t\bar{X}_tdt+\dot{C}_tdt\\
     &+B_t\big(\big(-\frac{(b_\alpha)^2}{c_\alpha} ((A_t+B_t)\bar{X}_t +C_t) +(b_X + 
     b_\mu)  \bar{X}_t\big) dt\big),
        \end{aligned}
    \end{equation}
    \normalsize
    where in the second equality we plugged in $dX_t$ and $d \bar{X}_t$ terms using the forward equation in~\eqref{eq:MI_FBSDE} and in the third equality we plugged in $Y_t$ and $\bar Y_t$ ansatz forms to have the final form. Finally, we can match the terms in equation~\eqref{eq:ansatz_backward} with the backward equation in~\eqref{eq:MI_FBSDE} to end up with the ODEs for the $A_t$, $B_t$ and $C_t$. The dynamics of $\bar{X}_t$ is also acquired after plugging in the ansatz for $Y_t$ in the forward equation in~\eqref{eq:MI_FBSDE} and taking the expectation.
\end{proof}

\begin{theorem}[Existence \& Uniqueness of a solution]
\label{the:exist_mi}
If the following condition holds, then there exists a unique MI-MFNE: 
\small
\begin{equation*}
\begin{aligned}
    &b_\mu \frac{-c_X\big(e^{(\delta^{+}-\delta^{-})(T-t)}-1\big)-c_T\big(\delta^{+} e^{(\delta^{+}-\delta^{-})(T-t)}-\delta^{-}\big)}{\big(\delta^{-} e^{(\delta^{+}-\delta^{-})(T-t)}-\delta^{+}\big)-c_T \frac{(b_\alpha)^2}{c_\alpha}\big(e^{(\delta^{+}-\delta^{-})(T-t)}-1\big)}\\
    &
    \qquad+c_\mu\geq 0, \qquad \forall t\in[0,T],
\end{aligned}
\end{equation*}
\normalsize
where $\delta^{\pm}=  b_X \pm \sqrt{b_X^2 + c_X (b_\alpha)^2/c_\alpha}$. 
\end{theorem}

\begin{proof}
From Proposition~\ref{prop:mi_FBODE}, we know that the MI-MFNE is characterized by solving the given FBODE system. Therefore, we can focus existence and uniqueness result of the FBODE system given in Proposition~\ref{prop:mi_FBODE}.
    We first realize that both the equations for $(A_t)_t$ and $(B_t)_t$ are Riccati equations. The Riccati equation for $(A_t)_t$~\eqref{eq:FBODE_A} has a unique solution which is bounded and continuous under the given model parameter assumptions ($c^{\rm C}_\alpha, c^{\rm C}_X, c^{\rm C}_\mu, c^{\rm C}_T>0$, and $b^{\rm C}_\alpha, b^{\rm C}_X, b^{\rm C}_\mu \in \RR$) and it can be explicitly written as:
    \small
    \begin{equation*}
        A_t = \frac{-c_X\big(e^{(\delta^{+}-\delta^{-})(T-t)}-1\big)-c_T\big(\delta^{+} e^{(\delta^{+}-\delta^{-})(T-t)}-\delta^{-}\big)}{\big(\delta^{-} e^{(\delta^{+}-\delta^{-})(T-t)}-\delta^{+}\big)-c_T \frac{(b_\alpha)^2}{c_\alpha}\big(e^{(\delta^{+}-\delta^{-})(T-t)}-1\big)},
    \end{equation*}
    \normalsize
    for all $t\in[0,T]$ and where $\delta^{+}$ and $\delta^{-}$ are as introduced in Theorem~\ref{the:exist_mi}~\cite[Chapter 2]{CarmonaDelarue_book_I},~\cite{JACOBSON1970258}. Once the first Riccati equation is solved, the solution $(A_t)_t$ can be plugged in the Riccati equation for $(B_t)_t$~\eqref{eq:FBODE_B}. Then, under the condition given in Theorem~\ref{the:exist_mi}, there exists a unique continuous solution to the Riccati equation for $(B_t)_t$,~\cite{FREILING2002243}~\cite{JACOBSON1970258}. After the unique and continuous solutions to equations~\eqref{eq:FBODE_A} and~\eqref{eq:FBODE_B} are found, they can be plugged into the first order linear differential equation for $(C_t)_t$~\eqref{eq:FBODE_C}. Since the coefficients are continuous, this ODE has a unique continuous solution (see e.g., Picard-Lindel\"of Theorem). Finally, we can plug in $(A_t)_t, (B_t)_t$, and $(C_t)_t$ in the differential equation for $(\bar{X}_t)_t$. Since these are continuous, the coefficients of the first order linear ODE for $(\bar{X}_t)_t$ are continuous, which gives the existence and uniqueness for~\eqref{eq:FBODE_BarX}. In turn, these conclude the existence and uniqueness of the solution of the FBODE system.
\end{proof}

\subsection{Mixed population}

\begin{theorem}[FBSDE characterization of equilibria]
    Controls $\hat{\balpha}^{\rm NC}$ and $\hat{\balpha}^{\rm C}$ (respectively for the non-cooperative and cooperative agents) are MP-MFE control profiles (see Definition~\ref{def:mixedpop_mfg_nash}) if and only if:
    \small
    \begin{equation}
    \label{eq:mp-BR-formula}
        \hat \alpha^{\rm NC}_t
        = -\frac{b^{\rm NC}_\alpha Y^{\rm NC}_t}{c^{\rm NC}_\alpha}, \qquad \hat \alpha^{\rm C}_t
        = -\frac{b^{\rm C}_\alpha Y^{\rm C}_t}{c^{\rm C}_\alpha},   
    \end{equation}
    \normalsize
    where $(\bX^{\rm NC}, \bY^{\rm NC}, \bZ^{\rm NC}, \bX^{\rm C}, \bY^{\rm C}, \bZ^{\rm C})=$ $(X^{\rm NC}_t, Y^{\rm NC}_t,$ \\$ Z^{\rm NC}_t, X^{\rm C}_t, Y^{\rm C}_t, Z^{\rm C}_t)_{t\in [0,T]}$ solve the following forward-backward stochastic differential equation (FBSDE) system:
    \small
    \begin{equation}
    \label{eq:mp_fbsde}
    \begin{aligned}
        &dX^{\rm NC}_t = \Big(-\frac{(b^{\rm NC}_\alpha)^2}{c^{\rm NC}_\alpha} Y^{\rm NC}_t +b^{\rm NC}_X X^{\rm NC}_t\\
        &+ b^{\rm NC}_\mu  (p \bar X_t^{\rm NC} +(1-p)\bar{X}^{\rm C}_t) \Big) dt + \sigma dW^{\rm NC}_t,
        \\[1mm]
        &dY^{\rm NC}_t = -\big(b^{\rm NC}_X Y^{\rm NC}_t + c^{\rm NC}_X X^{\rm NC}_t \big)dt+ Z^{\rm NC}_t dW^{\rm NC}_t,   \\[1mm]
        &dX^{\rm C}_t = \Big(-\frac{(b^{\rm C}_\alpha)^2}{c^{\rm C}_\alpha} Y^{\rm C}_t +b^{\rm C}_X X^{\rm C}_t\\
        &+ b^{\rm C}_\mu  (p \bar X_t^{\rm NC} +(1-p)\bar{X}^{\rm C}_t) \Big) dt + \sigma dW^{\rm C}_t,
        \\[1mm]
        &dY^{\rm C}_t = -\big(b^{\rm C}_X Y^{\rm C}_t + c^{\rm C}_X X^{\rm C}_t + b^{\rm C}_\mu (1-p)\bar Y_t^{\rm C}\\
        &+2c_\mu^{\rm C}(1-p)(p\bar{X}_t^{\rm NC}+(1-p)\bar X^{\rm C}_t) \big)dt+ Z^{\rm C}_t dW^{\rm C}_t,\\[1mm]  
        & X^{\rm NC}_0 \sim \mu^{\rm NC}_0, X^{\rm C}_0 \sim \mu^{\rm C}_0, \ Y^{\rm NC}_T =c^{\rm NC}_T X^{\rm NC}_T,  Y^{\rm C}_T =c^{\rm C}_T X^{\rm C}_T.
    \end{aligned}
    \end{equation}
    \normalsize
\end{theorem}

\begin{proof} 
For the non-cooperative and cooperative representative agents we can write their respective Hamiltonians as follows:
\small
\begin{equation*}
    \begin{aligned}
        &H^{\rm NC}(t, x^{\rm NC}, \alpha^{\rm NC}, \bar x^{\rm NC}, \bar x^{\rm C}, y^{\rm NC})= \big(b^{\rm NC}_\alpha \alpha^{\rm NC} + b^{\rm NC}_X x^{\rm NC} \\
        &\qquad+ b^{\rm NC}_\mu  \big(p\bar{x}^{\rm NC}_t + (1-p) \bar{x}_t^{\rm C}\big)\big)y^{\rm NC} +\frac{c^{\rm NC}_\alpha}{2} (\alpha^{\rm NC})^2 \\
        &\qquad + \frac{c^{\rm NC}_X}{2}(x^{\rm NC})^2+\frac{c^{\rm NC}_\mu}{2}\left(p \bar x^{\rm NC} +(1-p)\bar{x}^{\rm C}\right)^2  \\
        &H^{\rm C}(t, x^{\rm C}, \alpha^{\rm C}, \bar x^{\rm C}, \bar x^{\rm NC}, y^{\rm C})= \big(b^{\rm C}_\alpha \alpha^{\rm C} + b^{\rm C}_X x^{\rm C} \\
        &\qquad+ b^{\rm C}_\mu  \big(p\bar{x}^{\rm NC}_t + (1-p) \bar{x}_t^{\rm C}\big)\big)y^{\rm NC} +\frac{c^{\rm C}_\alpha}{2} (\alpha^{\rm C})^2 \\
        &\qquad + \frac{c^{\rm C}_X}{2}(x^{\rm C})^2+\frac{c^{\rm C}_\mu}{2}\left(p \bar x^{\rm NC} +(1-p)\bar{x}^{\rm C}\right)^2.
    \end{aligned}
\end{equation*}
\normalsize
Then, the equilibrium controls will be given as the minimizers of the Hamiltonian:
\small
    \begin{equation}
        \hat \alpha^{\rm NC}_t
        = -\frac{b^{\rm NC}_\alpha Y^{\rm NC}_t}{c^{\rm NC}_\alpha}, \qquad \hat \alpha^{\rm C}_t
        = -\frac{b^{\rm C}_\alpha Y^{\rm C}_t}{c^{\rm C}_\alpha}.
    \end{equation}
    \normalsize
    Here $\bY^{\rm NC}$ and $\bY^{\rm NC}$ will be determined as the solution of a forward-backward stochastic differential equation system. In order to construct this system, we realize that the representative non-cooperative agent solves a mean field game given the mean field of the cooperative agents, $\bar{\bX}^{\rm C}$, and the representative cooperative agent solves a mean field control given the mean field of the non-cooperative agents, $\bar{\bX}^{\rm NC}$. Following this idea, we can write the FBSDE system for the representative non-cooperative agent as follows (e.g.~\cite[Chapter 3]{CarmonaDelarue_book_I}:
    \small
    \begin{equation}
        \begin{aligned}
        &dX^{\rm NC}_t = \Big(-\frac{(b^{\rm NC}_\alpha)^2}{c^{\rm NC}_\alpha} Y^{\rm NC}_t +b^{\rm NC}_X X^{\rm NC}_t\\
        &+ b^{\rm NC}_\mu  (p \bar X_t^{\rm NC} +(1-p)\bar{X}^{\rm C}_t) \Big) dt + \sigma dW^{\rm NC}_t,
        \\[1mm]
        &dY^{\rm NC}_t = -\big(b^{\rm NC}_X Y^{\rm NC}_t + c^{\rm NC}_X X^{\rm NC}_t \big)dt+ Z^{\rm NC}_t dW^{\rm NC}_t,   \\[1mm]
        &X^{\rm NC}_0 \sim \mu^{\rm NC}_0,\ Y^{\rm NC}_T =c^{\rm NC}_T X^{\rm NC}_T.
        \end{aligned}
    \end{equation}
    \normalsize
    Here, the forward component gives the state dynamics under the equilibrium control $\balpha^{\rm NC}$ (i.e., the control form at the equilibrium~\eqref{eq:mp-BR-formula} is plugged in the state dynamics. The backward component (i.e., the dynamics of the adjoint process $\bY$) is written as $dY^{\hat\balpha}_t = -\partial_x H^{\rm NC}(t, X_t^{\rm NC}, \alpha_t^{\rm NC}, \bar X_t^{\rm NC}, \bar X_t^{\rm C}, Y_t^{\rm NC}) dt + Z^{\rm NC}_t dW_t$ with terminal condition $Y_T^{\rm NC} = \partial_x \frac{c^{\rm NC}_T}{2} (X^{\rm NC}_T)^2$ by using the stochastic Pontryagin maximum principle.
    Similarly, we can write the FBSDE system for the representative cooperative agent as follows:
    \small
    \begin{equation}
        \begin{aligned}
        &dX^{\rm C}_t = \Big(-\frac{(b^{\rm C}_\alpha)^2}{c^{\rm C}_\alpha} Y^{\rm C}_t +b^{\rm C}_X X^{\rm C}_t\\
        &+ b^{\rm C}_\mu  (p \bar X_t^{\rm NC} +(1-p)\bar{X}^{\rm C}_t) \Big) dt + \sigma dW^{\rm C}_t,
        \\[1mm]
        &dY^{\rm C}_t = -\big(b^{\rm C}_X Y^{\rm C}_t + c^{\rm C}_X X^{\rm C}_t + b^{\rm C}_\mu (1-p)\bar Y_t^{\rm C}\\
        &+2c_\mu^{\rm C}(1-p)(p\bar{X}_t^{\rm NC}+(1-p)\bar X^{\rm C}_t) \big)dt+ Z^{\rm C}_t dW^{\rm C}_t,\\[1mm]  
        & X^{\rm C}_0 \sim \mu^{\rm C}_0, \ Y^{\rm C}_T =c^{\rm C}_T X^{\rm C}_T.
        \end{aligned}
    \end{equation}
    \normalsize
    Here, the forward component gives the state dynamics under the equilibrium control $\balpha^{\rm NC}$ (i.e., the control form at the equilibrium~\eqref{eq:mp-BR-formula} is plugged in the state dynamics. The backward component (i.e., the dynamics of the adjoint process $\bY$) is written as $dY^{\hat\balpha}_t = -\Big(\partial_x H^{\rm C}(t, X_t^{\rm C}, \alpha_t^{\rm C}, \bar X_t^{\rm C}, \bar X_t^{\rm NC}, Y_t^{\rm C}) + \partial_{\bar{x}^{\rm C}} H^{\rm N}(t, X_t^{\rm C}, \alpha_t^{\rm C}, \bar X_t^{\rm C}, \bar X_t^{\rm NC}, Y_t^{\rm C})\Big) dt + Z^{\rm C}_t dW_t$ with terminal condition $Y_T^{\rm C} = \partial_x \frac{c^{\rm C}_T}{2} (X^{\rm C}_T)^2$ by using the stochastic Pontryagin maximum principle. These two FBSDE systems are coupled through the mean fields $\bar{\bX}^{\rm NC}$ and $\bar{\bX}^{\rm C}$, in this way we end up with our final FBSDE system with 2 forward and 2 backward components. 
\end{proof}

\begin{proposition}[FBODE characterization of equilibria]
\label{prop:mp_FBODE}
    Assume there exists an $\RR^{2\times2} \times \RR^{2\times2} \times \RR^{2\times1} \times \RR^{2\times1}$-valued function $t \mapsto (A_t, B_t, C_t, \bar{X}_t)$ solving the following forward-backward ordinary differential equation (FBODE) system:
\small
\begin{align}
        &{\dot A_t} + A_t M_2 A_t + M_1 A_t + A_t M_1 - M_4 =0,\label{eq:FBODE_A_MP} \\[1mm]
        &\dot B_t + B_t M_2 B_t  +A_t M_2 B_t + B_t M_2 A_t\label{eq:FBODE_B_MP}\\
        &+ (M_1-M_6) B_t + B_t (M_1+M_3) + A_t M_3 - M_5=0,\nonumber\\[1mm]
       &\dot C_t +((A_t +B_t)M_2 + M_1-M_6) C_t =0,  \label{eq:FBODE_C_MP}\\[1mm]
       &\dot{\bar{X}}_t = (M_1 + M_2 A_t + M_2 B_t +M_3)\bar{X}_t + M_2C_t,\label{eq:FBODE_BarX_MP}\\[1mm]
       &A_T = \begin{bmatrix}
c_T^{\rm NC} &0\\
0 & c_T^{\rm C}\end{bmatrix} B_T = \begin{bmatrix}
0 &0\\
0 & 0
\end{bmatrix}, C_T=\begin{bmatrix}
0 \\
0 
\end{bmatrix}\label{eq:FBODE_boundarycond_MP},
\end{align}
\small
where $M_1=\begin{bmatrix}
b_X^{\rm NC} & 0\\
0&b_X^{\rm C} 
\end{bmatrix}$, 
$M_2=\begin{bmatrix}
-(b_\alpha^{\rm NC})^2/ c_\alpha^{\rm NC} & 0\\
0&-(b_\alpha^{\rm C})^2/ c_\alpha^{\rm C}
\end{bmatrix}$,\\[2mm] 
$M_3=\begin{bmatrix}
p b_\mu^{\rm NC}  & (1-p) b_\mu^{\rm NC}\\
p b_\mu^{\rm C}   & (1-p) b_\mu^{\rm C}
\end{bmatrix}$,
$M_4=\begin{bmatrix}
-c_X^{\rm NC} & 0\\
0&-c_X^{\rm C} 
\end{bmatrix}$, \\[2mm]
$M_5=\begin{bmatrix}
0 & 0\\
2(1-p)p c_\mu^{\rm C}&2(1-p)^2 c_\mu^{\rm C}
\end{bmatrix}$, 
$M_6=\begin{bmatrix}
0 & 0\\
0&b_\mu^{\rm C}(1-p)
\end{bmatrix}$.
\normalsize
Then, 
\small
\begin{equation}
    \label{eq:mp_br_formula}
\hat \alpha_t= - K(A_t X_t + B_t \bar{X}_t+C_t),
\end{equation}
\normalsize
is the MP-MFE control, where 
\small $K=\begin{bmatrix}
-b_\alpha^{\rm NC}/ c_\alpha^{\rm NC} & 0\\
0&-b_\alpha^{\rm C}/ c_\alpha^{\rm C}
\end{bmatrix}$,\vskip1mm
$X_t=\begin{bmatrix}
X_t^{\rm NC}\\
X_t^{\rm C}
\end{bmatrix}$,
$\bar{X}_t=\begin{bmatrix}
\bar{X}_t^{\rm NC}\\
\bar{X}_t^{\rm C}
\end{bmatrix}$,
$\hat{\alpha}_t=\begin{bmatrix}
\hat{\alpha}_t^{\rm NC}\\
\hat{\alpha}_t^{\rm C}
\end{bmatrix}$. 
\normalsize
\end{proposition}
\vskip3mm 
We would like to emphasize that $A_t$, $B_t$ and $C_t$ that are introduced in Propositions~\ref{prop:mi_FBODE} and~\ref{prop:mp_FBODE} are different functions. In the Proposition~\ref{prop:mi_FBODE}, these functions are $\RR$-valued; however in Proposition~\ref{prop:mp_FBODE}, $A_t$ and $B_t$ are matrix ($\RR^{2\times2}$)-valued functions and $C_t$ is a vector ($\RR^{2\times1}$)-valued function.

\begin{proof}
    We first write the FBSDE system~\eqref{eq:mp_fbsde} in the matrix form as follows:
    \small
    \begin{equation}
    \label{eq:mp_fbsde_matrix}
    \begin{aligned}
        &dX_t = (M_1 X_t + M_2 Y_t + M_3 \bar{X}_t) dt + \Sigma d W_t,
        \\[1mm]
        &dY_t = (-M_1 Y_t + M_4 X_t + M_5\bar{X}_t + M_6\bar{Y}_t)dt+ Z_t dW_t,
        \end{aligned}
    \end{equation}
    \normalsize
    where $X_t, \bar{X}_t, M_1, M_2, M_3, M_4, M_5, M_6$ are introduced in Proposition~\ref{prop:mp_FBODE} and \small
     $Y_t=\begin{bmatrix}
    Y_t^{\rm NC} \\
    Y_t^{\rm C} 
    \end{bmatrix}$,
    $\Sigma=\begin{bmatrix}
    \sigma^{\rm NC} & 0\\
    0& \sigma^{\rm C} 
    \end{bmatrix}$, $W_t=\begin{bmatrix}
    W_t^{\rm NC} \\
    W_t^{\rm C} 
    \end{bmatrix}$.\normalsize Then, we propose the ansatz $Y_t=A_tX_t + B_t \bar{X}_t+C_t$ and take the derivative and plug in first the $dX_t$, $d\bar{X}_t$ and second $Y_t$, $\bar{Y}_t$ ansatz forms to conclude:
    \small
    \begin{equation}
    \label{eq:mp_ansatz_backward}
        \begin{aligned}
            dY_t =& \dot{A}_t X_tdt + A_t dX_t + \dot{B}_t\bar{X}_tdt + B_td\bar{X}_t +\dot{C}_tdt\\
            =&\dot{A}_t X_tdt + A_t \Big((M_1 X_t + M_2(A_tX_t + B_t \bar{X}_t+C_t)  \\
            &+ M_3 \bar{X}_t) dt + \Sigma d W_t \Big)+ \dot{B}_t\bar{X}_tdt+\dot{C}_tdt\\
             &+ B_t\Big(((M_1+M_3) \bar{X}_t + M_2 (A_t+ B_t) \bar{X}_t+C_t) dt\Big). 
        \end{aligned}
    \end{equation}
    \normalsize
    Finally, we match the terms in equation~\eqref{eq:mp_ansatz_backward} with the backward equation in~\eqref{eq:mp_fbsde_matrix} to end up with the ODEs for the $A_t$, $B_t$ and $C_t$. The dynamics of $\bar{X}_t$ is acquired after plugging in the ansatz for $Y_t$ in the forward equation in~\eqref{eq:mp_fbsde_matrix} and taking the expectation.
\end{proof}

\begin{assumption}
\label{assu:exist_mp}
Let $M_{21}(t) := -A_t M_3 + M_5$, $M_{11}(t) := M_2A_t+ M_1 + M_3$, and $M_{22}(t) := -A_t M_2 -M_1 + M_6$.

\noindent For some matrices $E\in \mathbb{C}^{2\times 2}$ with $E^*=E$, $F\in\mathbb{C}^{2\times 2}$
with
\small
\begin{equation*}
     L(t)= \begin{bmatrix}
         EM_{11}(t) +F M_{21}(t) & E M_2 + M^{\top}_{11}(t)F + F M_{22}(t)\\
         0 & M_2 F
     \end{bmatrix}
 \end{equation*}

 \normalsize
\noindent the condition $L(t)+L^*(t)\leq0$ holds for all $t\in[0,T]$ and $E>0$. Here $G^*$ denotes the complex conjugate of complex valued matrix $G$.
\end{assumption}
\begin{theorem}[Existence of a solution]
\label{the:exist_mp}
 If the Assumption~\ref{assu:exist_mp} holds, then there exists an MP-MFE.
\end{theorem}
\begin{proof}
    Similar to the proof of Theorem~\ref{the:exist_mi}, in order to show the existence of the MP-MFE, we can focus on the existence of the solution of the FBODE system given in Proposition~\ref{prop:mp_FBODE} that characterizes the solution of the MP-MFE. We again realize that the differential equations for $(A_t)_t$~\eqref{eq:FBODE_A_MP} and $(B_t)_t$~\eqref{eq:FBODE_B_MP} are matrix Riccati equations and the equation~\eqref{eq:FBODE_A_MP} has a unique and continuous solution that is bounded under the given model parameter assumptions. \cite{JACOBSON1970258,FREILING2002243} Then we can plug in the $(A_t)_t$ matrix in Riccati equation~\eqref{eq:FBODE_B_MP} to solve for $(B_t)_t$. This equation is a nonsymmetric Riccati equation and under the assumption~\ref{assu:exist_mp}, it has a continuous solution (see~\cite[Theorem 3.11]{FREILING2002243}). Then we can plug in $(A_t)_t$ and $(B_t)_t$ in the linear ODE system~\eqref{eq:FBODE_C_MP} which has a unique continuous solution $(C_t)_t$ since the coefficients of the linear ODE are continuous. Then by plugging in $(A_t)_t$, $(B_t)_t$ and $(C_t)_t$ in the linear ODE~\eqref{eq:FBODE_BarX_MP} we can find the unique solution $(\bar X_t)_t$ since the coefficients of the linear ODE are continuous. Different than the Theorem~\ref{the:exist_mi}, we only conclude the existence results, since the nonsymmetric Riccati equation~\ref{eq:FBODE_B_MP} does not have a uniqueness result for the given condition.
\end{proof}

\section{Conclusion \& Future Work}

In this paper, we have proposed two families of models to study large populations of strategic agents with a combination of cooperative and non-cooperative behavior. We first presented the models with finitely many agents, and then we presented the mean field models when the number of agents goes to infinity. For each type of mean field model, we proved optimality conditions based on Pontryagin stochastic maximum principle, using forward-backward stochastic differential equations of McKean-Vlasov type and ordinary differential equations, for which we showed existence (for both mixed individual and mixed population models) and uniqueness (for the mixed individual model) of solutions under suitable conditions. 

As future work, firstly we plan to theoretically and numerically analyze the effect of $\lambda$ and $p$. Secondly, we plan to study general models beyond linear-quadratic structure, and applications to the tragedy of the commons. Finally, we plan to analyze the convergence of $\epsilon$-Nash equilibrium to the mean field equilibrium in both settings.

\bibliographystyle{ieeetr}

\begin{thebibliography}{10}
	
	\bibitem{Bensoussan_Book}
	A.~Bensoussan, J.~Frehse, and P.~Yam, ``Mean field games and mean field type
	control theory,'' {\em Springer Briefs in Mathematics}, 2013.
	
	\bibitem{CarmonaDelarue_book_I}
	R.~Carmona and F.~Delarue, {\em Probabilistic theory of mean field games with
		applications. {I}}, vol.~83 of {\em Probability Theory and Stochastic
		Modelling}.
	\newblock Cham: Springer, 2018.
	\newblock Mean field FBSDEs, control, and games.
	
	\bibitem{CarmonaDelarue_book_II}
	R.~Carmona and F.~Delarue, {\em Probabilistic theory of mean field games with
		applications. {II}}, vol.~84 of {\em Probability Theory and Stochastic
		Modelling}.
	\newblock Springer, Cham, 2018.
	\newblock Mean field games with common noise and master equations.
	
	\bibitem{carmona2015mean}
	R.~A. Carmona, J.~P. Fouque, and L.~H. Sun, ``Mean field games and systemic
	risk,'' {\em Communications in Mathematical Sciences}, vol.~13, no.~4,
	pp.~911--933, 2015.
	
	\bibitem{cardaliaguet2018mean}
	P.~Cardaliaguet and C.-A. Lehalle, ``Mean field game of controls and an
	application to trade crowding,'' {\em Mathematics and Financial Economics},
	vol.~12, pp.~335--363, 2018.
	
	\bibitem{carmona2021applications}
	R.~Carmona, ``Applications of mean field games in financial engineering and
	economic theory,'' in {\em Mean Field Games on Agent Based Models to Nash
		Equilibria, AMS 2020}, pp.~165--218, American Mathematical Society, 2021.
	
	\bibitem{carmona2023deep}
	R.~Carmona and M.~Lauri{\`e}re, ``Deep learning for mean field games and mean
	field control with applications to finance,'' {\em Machine Learning and Data
		Sciences for Financial Markets: A Guide to Contemporary Practices}, p.~369,
	2023.
	
	\bibitem{achdou2014partial}
	Y.~Achdou, F.~J. Buera, J.-M. Lasry, P.-L. Lions, and B.~Moll, ``Partial
	differential equation models in macroeconomics,'' {\em Philosophical
		Transactions of the Royal Society A: Mathematical, Physical and Engineering
		Sciences}, vol.~372, no.~2028, p.~20130397, 2014.
	
	\bibitem{achdou2022income}
	Y.~Achdou, J.~Han, J.-M. Lasry, P.-L. Lions, and B.~Moll, ``Income and wealth
	distribution in macroeconomics: A continuous-time approach,'' {\em The review
		of economic studies}, vol.~89, no.~1, pp.~45--86, 2022.
	
	\bibitem{advertisement}
	R.~Carmona and G.~Dayan\i{}kl\i{}, ``Mean field game model for an advertising
	competition in a duopoly,'' {\em International Game Theory Review}, vol.~23,
	no.~04, p.~2150024, 2021.
	
	\bibitem{elie2020contact}
	R.~Elie, E.~Hubert, and G.~Turinici, ``Contact rate epidemic control of
	covid-19: an equilibrium view,'' {\em Mathematical Modelling of Natural
		Phenomena}, vol.~15, p.~35, 2020.
	
	\bibitem{aurell2022optimal}
	A.~Aurell, R.~Carmona, G.~Dayan{\i}kl{\i}, and M.~Lauri\`ere, ``Optimal
	incentives to mitigate epidemics: a stackelberg mean field game approach,''
	{\em SIAM Journal on Control and Optimization}, vol.~60, no.~2,
	pp.~S294--S322, 2022.
	
	\bibitem{olmez2022modeling}
	S.~Y. Olmez, S.~Aggarwal, J.~W. Kim, E.~Miehling, T.~Ba{\c{s}}ar, M.~West, and
	P.~G. Mehta, ``Modeling presymptomatic spread in epidemics via mean-field
	games,'' in {\em 2022 American control conference (ACC)}, pp.~3648--3655,
	IEEE, 2022.
	
	\bibitem{kolokoltsov2016mean}
	V.~N. Kolokoltsov and A.~Bensoussan, ``Mean-field-game model for botnet defense
	in cyber-security,'' {\em Applied Mathematics \& Optimization}, vol.~74,
	pp.~669--692, 2016.
	
	\bibitem{gueant2011mean}
	O.~Gu{\'e}ant, P.~L. Lions, and J.-M. Lasry, ``Mean field games and
	applications,'' {\em Paris-Princeton Lectures on Mathematical Finance 2010},
	2011.
	
	\bibitem{chan2017fracking}
	P.~Chan and R.~Sircar, ``Fracking, renewables, and mean field games,'' {\em
		SIAM Review}, vol.~59, no.~3, pp.~588--615, 2017.
	
	\bibitem{alasseur2020extended}
	C.~Alasseur, I.~Ben~Taher, and A.~Matoussi, ``An extended mean field game for
	storage in smart grids,'' {\em Journal of Optimization Theory and
		Applications}, vol.~184, pp.~644--670, 2020.
	
	\bibitem{carbon_StackelbergMFG}
	R.~Carmona, G.~Dayan{\i}kl{\i}, and M.~Lauri{\`e}re, ``Mean field models to
	regulate carbon emissions in electricity production,'' {\em Dynamic Games and
		Applications}, vol.~12, no.~3, pp.~897--928, 2022.
	
	\bibitem{dayanikli2024multipopulation}
	G.~Dayanikli and M.~Lauriere, ``Multi-population mean field games with multiple
	major players: Application to carbon emission regulations,'' in {\em 2024
		American Control Conference (ACC)}, pp.~5075--5081, IEEE, 2024.
	
	\bibitem{ostrom1990governing}
	E.~Ostrom, {\em Governing the commons: The evolution of institutions for
		collective action}.
	\newblock Cambridge university press, 1990.
	
	\bibitem{ostrom1999coping}
	E.~Ostrom, ``Coping with tragedies of the commons,'' {\em Annual review of
		political science}, vol.~2, no.~1, pp.~493--535, 1999.
	
	\bibitem{carmona2023nash}
	R.~Carmona, G.~Dayan{\i}kl{\i}, F.~Delarue, and M.~Lauri\`ere, ``From nash
	equilibrium to social optimum and vice versa: a mean field perspective,''
	{\em arXiv preprint arXiv:2312.10526}, 2023.
	
	\bibitem{angiuli2023reinforcementmixedmfcg}
	A.~Angiuli, N.~Detering, J.-P. Fouque, J.~Lin, {\em et~al.}, ``Reinforcement
	learning algorithm for mixed mean field control games,'' {\em Journal of
		Machine Learning}, vol.~2, no.~2, 2023.
	
	\bibitem{Barreiro_Gomez_Duncan_Tembine_2020}
	J.~Barreiro-Gomez, T.~E. Duncan, and H.~Tembine, ``Co-opetitive
	linear-quadratic mean-field-type games,'' {\em IEEE Transactions on
		Cybernetics}, vol.~50, p.~5089–5098, Dec 2020.
	
	\bibitem{guo2023mesob}
	X.~Guo, L.~Li, S.~Nabi, R.~Salhab, and J.~Zhang, ``{MESOB}: Balancing
	equilibria \& social optimality,'' {\em arXiv preprint arXiv:2307.07911},
	2023.
	
	\bibitem{achdou2017mean}
	Y.~Achdou, M.~Bardi, and M.~Cirant, ``Mean field games models of segregation,''
	{\em Mathematical Models and Methods in Applied Sciences}, vol.~27, no.~01,
	pp.~75--113, 2017.
	
	\bibitem{cirant2015multi}
	M.~Cirant, ``Multi-population mean field games systems with neumann boundary
	conditions,'' {\em Journal de Math{\'e}matiques Pures et Appliqu{\'e}es},
	vol.~103, no.~5, pp.~1294--1315, 2015.
	
	\bibitem{djehiche2017mean}
	B.~Djehiche, A.~Tcheukam, and H.~Tembine, ``Mean-field-type games in
	engineering,'' {\em AIMS Electronics and Electrical Engineering}, vol.~1,
	no.~1, pp.~18--73, 2017.
	
	\bibitem{barreiro2021mean}
	J.~Barreiro-Gomez and H.~Tembine, {\em Mean-field-type Games for Engineers}.
	\newblock CRC Press, 2021.
	
	\bibitem{huang2007large}
	M.~Huang, P.~E. Caines, and R.~P. Malham{\'e}, ``Large-population cost-coupled
	lqg problems with nonuniform agents: Individual-mass behavior and
	decentralized $\varepsilon $-nash equilibria,'' {\em IEEE transactions on
		automatic control}, vol.~52, no.~9, pp.~1560--1571, 2007.
	
	\bibitem{bardi2012explicit}
	M.~Bardi, ``Explicit solutions of some linear-quadratic mean field games,''
	{\em Networks and heterogeneous media}, vol.~7, no.~2, pp.~243--261, 2012.
	
	\bibitem{bensoussan2016linear}
	A.~Bensoussan, K.~Sung, S.~C.~P. Yam, and S.-P. Yung, ``Linear-quadratic mean
	field games,'' {\em Journal of Optimization Theory and Applications},
	vol.~169, pp.~496--529, 2016.
	
	\bibitem{delarue2020selection}
	F.~Delarue and R.~F. Tchuendom, ``Selection of equilibria in a linear quadratic
	mean-field game,'' {\em Stochastic Processes and their Applications},
	vol.~130, no.~2, pp.~1000--1040, 2020.
	
	\bibitem{wang2020mean}
	B.-C. Wang, H.~Zhang, and J.-F. Zhang, ``Mean field linear--quadratic control:
	Uniform stabilization and social optimality,'' {\em Automatica}, vol.~121,
	p.~109088, 2020.
	
	\bibitem{huang2010large}
	M.~Huang, ``Large-population lqg games involving a major player: the nash
	certainty equivalence principle,'' {\em SIAM Journal on Control and
		Optimization}, vol.~48, no.~5, pp.~3318--3353, 2010.
	
	\bibitem{JACOBSON1970258}
	D.~Jacobson, ``New conditions for boundedness of the solution of a matrix
	riccati differential equation,'' {\em Journal of Differential Equations},
	vol.~8, no.~2, pp.~258--263, 1970.
	
	\bibitem{FREILING2002243}
	G.~Freiling, ``A survey of nonsymmetric riccati equations,'' {\em Linear
		Algebra and its Applications}, vol.~351-352, pp.~243--270, 2002.
	\newblock Fourth Special Issue on Linear Systems and Control.
	
\end{thebibliography}

\end{document}